\documentclass[12pt]{article}
\usepackage{amsmath}
\usepackage{amsthm}
\usepackage{amssymb}
\usepackage{mathdots}

\newtheorem{theorem}{Theorem}

\newtheorem{lemma}{Lemma}

\begin{document}

\vspace*{30px}

\begin{center}\Large
%%% Insert the title of your talk here %%%
\textbf{A New Generalized Cassini Determinant}
%\textbf{A New Cassini Determinant of Order n}
\bigskip\large

%%% Insert your first name and family name here %%%
Ivica Martinjak\\
University of Zagreb, Faculty of Science\\
Bijeni\v cka 32, HR-10000 Zagreb, Croatia\\
%imartnjak@phy.hr\\
and\\
Igor Urbiha\\
Polytechnic of Zagreb\\
Vrbik 8, HR-10000 Zagreb, Croatia\\
%igor.urbiha@tvz.hr

%%% Insert your institution here %%%

%%% Insert you city and country here %%%
\end{center}

%%% Give the abstract of your talk here %%%

\begin{abstract} 
In this paper we extend a notion of Cassini determinant to recently introduced hyperfibonacci sequences. We find $Q$-matrix for the $r$-th generation hyperfibonacci numbers and prove an explicit expression of the Cassini determinant for these sequences. 
\end{abstract}


\section{Introduction}

Given the second order recurrence relation, defined by 
\begin{eqnarray}\label{secondorderrecurr}
a_{n+2}=\alpha a_{n+1} + \beta a_{n},
\end{eqnarray}
where $\alpha$ and $\beta$ are constants, a sequence $(a_k)_{k \ge 0}$ is called a solution of (\ref{secondorderrecurr}) if its terms satisfy this recurrence. The set of all solutions of (\ref{secondorderrecurr}) forms a {\it linear space}, meaning that if $(a_k)_{k \ge 0}$ and $(b_k)_{k \ge 0}$ are two solutions then $(a_k+b_k)_{k \ge 0}$ is also a solution of (\ref{secondorderrecurr}). Furthermore, it holds true that for any constant $c$, $(ca_k)_{k \ge 0}$ is also a solution of (\ref{secondorderrecurr}). Using these basic properties of a linear space one can derive the identity  
\begin{eqnarray}\label{cassinipro}
a_{m}b_{m-1} - a_{m-1}b_{m} = (-\beta)^{m-1}(a_1b_0 - a_0b_1),
\end{eqnarray}
where $(a_k)_{k \ge 0}$ and $(b_k)_{k \ge 0}$ are two solutions of recurrence (\ref{secondorderrecurr}) \cite{Voll}.
When $\alpha=\beta=1$ and initial values of the terms are 0 and 1, respectively, the relation (\ref{secondorderrecurr}) defines the well known {\it Fibonacci sequence} $(F_k)_{k \ge 0}$. One can find more on this subject in a classic reference \cite{Vajda}. In case of the Fibonacci sequence relation (\ref{cassinipro}) reduces to 
\begin{eqnarray}\label{cassini}
F_{n-1}F_{n+1} - F_n^2 = (-1)^n
\end{eqnarray}
and it is called {\it Cassini identity} \cite{GKP, Mile, WeZe}. This relation can also be written in matrix form as 
\begin{eqnarray} \label{cassinimatrix}
\det
 \begin{pmatrix}
  F_{n} & F_{n+1}  \\
F_{n+1} & F_{n+2}\\
 \end{pmatrix}
=(-1)^n.
\end{eqnarray}

In this paper we study the {\it hyperfibonacci sequences} which are defined by the relation
\begin{eqnarray}\label{hyperFibo}
F_n^{(r)}  = \sum_{k=0} ^n F_k ^{(r-1)}, \enspace F_n^{(0)}= F_n, \enspace F_0^{(r)}=0, \enspace F_1^{(r)}=1,
\end{eqnarray}
where $r \in \mathbb{N}$ and $F_n$ is the $n$-th Fibonacci number. The number $F_n^{(r)}$ we shall call $n$-th hyperfibonacci number of {\it $r$-th generation}.
These sequences are recently introduced by Dill and Mez\H o \cite{DiMe}. Several interesting theoretical number and combinatorial properties of these sequences are already proven, including those available in \cite{CaZh}. Here we define the matrix 
\begin{eqnarray*}
A_{r,n}=
 \begin{pmatrix}
  F_{n}^{(r)} & F_{n+1}^{(r)} & \cdots & F_{n+r+1}^{(r)} \\
  F_{n+1}^{(r)} & F_{n+2}^{(r)} & \cdots & F_{n+r+2}^{(r)}  \\
  \vdots  & \vdots  & \ddots & \vdots  \\
   F_{n+r+1}^{(r)} &   F_{n+r+2}^{(r)} & \cdots &   F_{n+2r+2}^{(r)}
 \end{pmatrix}
\end{eqnarray*}
and we prove that $\det(A_{r,n})$ is an extension of (\ref{cassini}). Thus, we show that the generalization of the Cassini identity, expressed in a matrix form, holds true for the hyperfinonacci sequences.

\section{$Q$-matrix of the hyperfibonacci sequences}

According to the definition (\ref{hyperFibo}) obviously we have 
\begin{eqnarray} \label{obviousRec}
F_{n+1}^{(r)} = F_{n}^{(r)} + F_{n+1}^{(r-1)}.
\end{eqnarray}
In case $r=1$ the second term $ F_{n+1}^{(r-1)}$ is determined by the Fibonacci recurrence relation,
\begin{eqnarray*}
F_{n+3}^{(1)}&=& F_{n+2}^{(1)} + \Big(F_{n+2}^{(1)} - F_{n+1}^{(1)} \Big) + \Big(F_{n+1}^{(1)} - F_{n}^{(1)} \Big),
\end{eqnarray*}
thus we have
\begin{eqnarray} \label{basicRec1}
F_{n+3}^{(1)} &=& 2 F_{n+2}^{(1)} - F_{n}^{(1)}.
\end{eqnarray}
Now, iteratively using (\ref{basicRec1}) we derive the recurrence relation
\begin{eqnarray} \label{LemmaCase1}
F_{n+2}^{(1)} = F_{n+1}^{(1)} + F_{n}^{(1)} + 1,
\end{eqnarray}
\begin{eqnarray*}
F_{n+3}^{(1)} &=& F_{n+2}^{(1)} + 2 F_{n+1}^{(1)} - F_{n-1}^{(1)} - F_{n}^{(1)}\\
&=& F_{n+2}^{(1)} + F_{n+1}^{(1)} + 2F_{n}^{(1)} - F_{n-2}^{(1)}-F_{n-1}^{(1)}-F_{n}^{(1)}\\
&=& F_{n+2}^{(1)} + F_{n+1}^{(1)} + \cdots + F_{3}^{(1)} - F_{0}^{(1)} - F_{1}^{(1)} - \cdots - F_{n}^{(1)}\\
&=& F_{n+2}^{(1)} + F_{n+1}^{(1)} +1.
\end{eqnarray*}
When $r=2$ we use the same approach to get recurrence for the second generation of the hyperfibonacci numbers,
\begin{eqnarray}\label{LemmaRec2}
F_{n+2}^{(2)} = F_{n+1}^{(2)} + F_{n}^{(2)} + n+2. 
\end{eqnarray}
Namely, in this case the second term in (\ref{obviousRec}) is determined by obtained recurrence relation (\ref{LemmaCase1}). This means that again we can perform the $(n+1)$-step iterative procedure, this time using 
\begin{eqnarray}
F_{n+3}^{(2)} = 2 F_{n+2}^{(2)} - F_n^{(2)} + 1.
\end{eqnarray}
The fact that terms indexed 3 through $n$ cancel each other and that $(n+1)$ 1s remains, completes the proof of (\ref{LemmaRec2}). 

Recall that {\it polytopic numbers} are generalization of square and triangular numbers. These numbers can be represented by a regular geometrical arrangement of equally spaced points. The $n$-th regular $r$-topic number $P_n^{(r)}$ is equal to 
\begin{eqnarray}
P_n^{(r)}= \binom{n+r-1}{r}.
\end{eqnarray}  

When $r=3$, the $i$-th step of the iterative procedure described above results with an extra $i$, which sum to a triangular number $\binom{n+3}{2}$ after the final $(n+1)$st iteration. Furthermore, in the next case we add the $i$-th triangular number in $i$-th step of iteration. According to the properties of polytopic numbers, these numbers sum to the tetrahedral number $\binom{n+4}{3}$. In general, in the $i$-th step of the iteration we add $i$-th regular $(r-1)$-topic number and sum of these numbers after the final step of the procedure is the regular polytopic number $\binom{n+r}{r-1}$. Now we collect all this reasoning into the following

\begin{lemma} \label{figurateNumbers} The difference between $n$-th $r$-generation hyperfibonacci number and the sum of its two consecutive predecessors is $n$-th regular (r-1)-topic number,
\begin{eqnarray} \label{eqLemaHyperPoly}
F_{n+2}^{(r)}=F_{n+1}^{(r)}+F_{n}^{(r)} + \binom{n+r}{r-1}.
\end{eqnarray}
\end{lemma}

We can also write relation (\ref{eqLemaHyperPoly}) as
\begin{eqnarray*}
F_{n+2}^{(r)}=F_{n+1}^{(r)}+F_{n}^{(r)} + P_{n+2}^{(r-1)}.
\end{eqnarray*}

Hyperfibonacci sequences can be defined by the vector recurrence relation
\begin{eqnarray}\label{vectorRecurrence}
 \begin{pmatrix}
F_{n+1}^{(r)}\\
F_{n+2}^{(r)}\\
\vdots \\
F_{n+r+2}^{(r)}
 \end{pmatrix}
=
Q_{r +2}
 \begin{pmatrix}
F_{n}^{(r)}\\
F_{n+1}^{(r)}\\
\vdots \\
F_{n+r+1}^{(r)}
 \end{pmatrix}
\end{eqnarray}
where $Q_{r+2}$ is a square matrix 
\begin{eqnarray} \label{QmatrixDef}
Q_{r+2}= 
 \begin{pmatrix}
0 & 1 &  0 & \cdots & 0 & 0\\
0 & 0 & 1 &  \cdots & 0 & 0\\
\vdots & \vdots & \vdots  &\ddots  &\vdots   & \vdots  \\
0 & 0 & 0 &\cdots & 1 & 0\\
0 & 0 & 0  &\cdots & 0 & 1\\
q_1 & q_2 & q_3 &\cdots & q_{r+1} & q_{r+2}\\
 \end{pmatrix}.
\end{eqnarray}
In order to determine elements $q_1,\ldots,q_{r+2}$ we use the fact that terms from $-r$ through $0$ of the $r$-th generation hyperfibonacci numbers takes values 0,
$$
\ldots \pm 1, 0, 0, \ldots, 0, 1, r+1,\ldots
$$
This follows from Lemma \ref{figurateNumbers} since we have 
\begin{eqnarray}
\binom{(n-2)+r}{r-1} = \frac{n(n+1)(n+2)\cdots (n+r-2)}{(r-1)!}.
\end{eqnarray}
These expressions are obviously equal to 0 for $n=0,-1,\ldots,-r$. 

In particular, when $n=-r+2$ we get
\begin{eqnarray*}
 \begin{pmatrix}
0\\
0\\
\vdots \\
0\\
1\\
F_{2}^{(r)}
 \end{pmatrix} 
=
 \begin{pmatrix}
0 & 1 &  0 & \cdots & 0 & 0\\
0 & 0 & 1 &  \cdots & 0 & 0\\
\vdots & \vdots & \vdots  &\ddots  &\vdots   & \vdots  \\
0 & 0 & 0 &\cdots & 1 & 0\\
0 & 0 & 0  &\cdots & 0 & 1\\
q_1 & q_2 & q_3 &\cdots & q_{r+1} & q_{r+2}\\
 \end{pmatrix}
 \begin{pmatrix}
0\\
0\\
\vdots \\
0\\
0\\
1
 \end{pmatrix}
\end{eqnarray*}
meaning that
$
q_{r+2}= F_2^{(r)}.
$
In the same way we obtain relations for all elements of $Q_{r+2}$,
\begin{eqnarray*}
q_{r+2}&=&F_2^{(r)}\\
q_{r+1}&=&F_3^{(r)} - F_2^{(r)} q_{r+2}\\
q_{r}&=&F_4^{(r)} - F_3^{(r)} q_{r+2} -  F_2^{(r)} q_{r+1}\\
& &\cdots \\
q_{1}&=&F_{r+3}^{(r)} - F_{r+2}^{(r)} q_{r+2} - \cdots -  F_2^{(r)} q_{2}.
\end{eqnarray*}
This reasoning gives the next Theorem \ref{thmQmatrix}.

\begin{theorem} \label{thmQmatrix}
For the hypefibonacci sequences we have
\begin{eqnarray}
A_{r,n} = Q_{r+2}^{n} A_{r,0}.
\end{eqnarray}
\end{theorem}
\begin{proof}
Relation (\ref{vectorRecurrence}) can be written as $A_{r,n}=Q_{r+2}A_{r,n-1}$. Now the statement of theorem follows immediately,
\begin{eqnarray*}
A_{r,n}&=&Q_{r+2} A_{r,n-1}= Q_{r+2}^2A_{r,n-2}\\
 &=& Q_{r+2}^nA_{r,0}.
\end{eqnarray*}
\end{proof}
Elements $q_1,\ldots, q_{r+2}$ can be expressed explicitly. In particular, expressions for $q_r, q_{r+1}, q_{r+2}$ are
\begin{eqnarray*}
q_{r+2} &=& 1+r\\
q_{r+1} &= & 1 - \binom{r+1}{2}\\
q_r&=&\frac{r^3-7r}{6}.
\end{eqnarray*}
As an example we calculate hyperfibonacci numbers $F_3^{(2)}, F_4^{(2)},\ldots, F_9^{(2)}$ of the second generation, collected in the matrix $A_{2,3}$.
%\begin{example}
For the second generation of the hyperfibonacci sequences we have 
\begin{eqnarray*}
A_{2,0}&=&
\begin{pmatrix}
0 & 1  & 3 & 7 \\
1 & 3  & 7 & 14 \\
3 & 7 & 14 & 26 \\
7 & 14  & 26 & 46 
\end{pmatrix}\\
Q_{4}&=&
 \begin{pmatrix}
0 & 1  & 0 & 0 \\
0 & 0  & 1 & 0 \\
0 & 0 & 0 & 1 \\
1 & -1  & -2 & 3 
 \end{pmatrix},\\
\end{eqnarray*}
according to (\ref{hyperFibo}) and (\ref{QmatrixDef}).
Now we determine the matrix $A_{2,3}$ by Theorem \ref{thmQmatrix},
\begin{eqnarray*}
A_{2,3} &=& 
\begin{pmatrix}
0 & 1  & 0 & 0 \\
0 & 0  & 1 & 0 \\
0 & 0 & 0 & 1 \\
1 & -1  & -2 & 3 
\end{pmatrix}^3
\begin{pmatrix}
0 & 1  & 3 & 7 \\
1 & 3  & 7 & 14 \\
3 & 7 & 14 & 26 \\
7 & 14  & 26 & 46 
\end{pmatrix}
=
\begin{pmatrix}
7 & 14  & 26 & 46 \\
14 & 26  & 46 & 79 \\
26 & 46 & 79 & 133 \\
46 &  79 & 133 & 221 
\end{pmatrix}.
\end{eqnarray*}
Note that the eigenvalues of $Q_4$ are $ \phi, 1, 1, \bar{ \phi}$, 
where 
\begin{eqnarray*} \label{GoldenRatio}
 \phi = \frac{1+ \sqrt{5}}{2}, \enspace  \bar{ \phi} = \frac{1- \sqrt{5}}{2}.
\end{eqnarray*}

%\end{example}

A class of matrices (\ref{QmatrixDef}) has some further interesting properties. Here we point out that the determinant of such a matrix is $-1$. This is demonstrated in the following 

\begin{lemma} \label{Qdeterminant} For $r \in \mathbb{N}$ the determinant of a matrix $Q_{r+2}$ takes value $-1$,
\begin{eqnarray*}
\det ( Q_{r+2} ) =-1. 
\end{eqnarray*}
\end{lemma}
\begin{proof}
We prove this statement by means of comparing determinants of matrices $A_{r,-r}$ and $A_{r,-r-1}$,
\begin{eqnarray*}
A_{r,-r} = Q_{r+2}A_{r,-r-1}.
\end{eqnarray*}
For matrix $A_{r,-r-1}$ we have
\begin{eqnarray*}
\det(A_{r,-r-1}) &=& 
\det
\begin{pmatrix}
(-1)^r  & 0 & 0 &\cdots &0 \\
  0 & 0 & 0 &\cdots &1  \\
 \vdots   &\vdots   &\vdots & \iddots  &\vdots  \\
 0 &   0 & 1 &\cdots &   F_{r-2}^{(r)}\\
 0 &   1 & r+1 & \cdots &   F_{r-1}^{(r)}
\end{pmatrix}_{r \times r}\\
&=&
(-1)^r \det
\begin{pmatrix}
0  & 0 &  \cdots & 0 & 1 \\
  0 & 0 &  \cdots & 1  &r+1  \\
  \vdots  & \vdots    & \iddots & \vdots& \vdots  \\
 0 &   1 &  \cdots &  F_{r-3}^{(r)} &  F_{r-2}^{(r)}\\
 1 &  r+1 &  \cdots &  F_{r-2}^{(r)} &  F_{r-1}^{(r)}
\end{pmatrix}_{(r-1) \times (r-1)}\\
&=& (-1)^{r+1} (-1)^{\lfloor (r-1)/2 \rfloor+1}\\
&=& (-1)^{\lfloor  r/2 \rfloor}.
\end{eqnarray*}
On the other hand, $\det(A_{r,-r}) = (-1)^{\lfloor  r/2 \rfloor} $ which proves that
\begin{eqnarray}
\det(A_{r,-r}) = - \det(A_{r,-r-1}).
\end{eqnarray}
Now, the statement of lemma follows immediately by the Binet-Cauchy theorem.
\end{proof}

It is worth mentioning that in \cite{LeSh} authors give some properties of the $k$-generalized Fibonacci $Q$-matrix.

\section{Cassini identity in a matrix form}

\begin{lemma}
	\label{L:1}
	\begin{equation}
		\label{Eq:3}
\det
		\begin{pmatrix}
			F_{n}-1 & F_{n+1}-1 & F_{n+2}-1 \\
			F_{n+1}-1 & F_{n+2}-1 & F_{n+3}-1 \\
			F_{n+2}-1 & F_{n+3}-1 & F_{n+4}-1
		\end{pmatrix}
		=(-1)^n,\ n\geq 0.
	\end{equation}
\end{lemma}
\proof
Using the definition of Fibonacci numbers and elementary transformations on rows and columns of determinants we get
\[
\begin{array}{l}
\det
\begin{pmatrix}
F_{n}-1 & F_{n+1}-1 & F_{n+2}-1 \\
F_{n+1}-1 & F_{n+2}-1 & F_{n+3}-1 \\
F_{n+2}-1 & F_{n+3}-1 & F_{n+4}-1
\end{pmatrix}
\\
=
\det
\begin{pmatrix}
F_{n}-1 & F_{n+1}-1 & F_{n+2}-1 \\
F_{n+1}-1 & F_{n+2}-1 & F_{n+3}-1 \\
F_{n}+F_{n+1}-1 & F_{n+1}+F_{n+2}-1 & F_{n+2}+F_{n+3}-1
\end{pmatrix}
\\[1.7em]%%%%%%%%%%%%%%%%%%%
=
\det
		\begin{pmatrix}
F_{n}-1 & F_{n+1}-1 & F_{n+2}-1 \\
F_{n+1}-1 & F_{n+2}-1 & F_{n+3}-1 \\
1 & 1 & 1
\end{pmatrix}

=
\det
		\begin{pmatrix}
F_{n}-1 & F_{n+1}-1 & F_{n}+F_{n+1}-1 \\
F_{n+1}-1 & F_{n+2}-1 & F_{n+1}+F_{n+2}-1 \\
1 & 1 & 1
\end{pmatrix} \\[1.7em]%%%%%%%%%%%%%%%%%%%

=
\det
		\begin{pmatrix}
F_{n}-1 & F_{n+1}-1 & 1 \\
F_{n+1}-1 & F_{n+2}-1 & 1 \\
1 & 1 & -1
\end{pmatrix}
=
\det
		\begin{pmatrix}
F_{n} & F_{n+1} & 0 \\
F_{n+1} & F_{n+2} & 0 \\
1 & 1 & -1
\end{pmatrix}
\\
=
-(F_nF_{n+2}-F_{n+1}^2)=(-1)^n\qed
\end{array}
\]

\begin{lemma} For the first generation of hyperfibonacci sequences $(F_n^{(1)})_{n \ge 0}$
\[
\det
		\begin{pmatrix}
F_{n}^{(1)} & F_{n+1}^{(1)} & F_{n+2}^{(1)} \\
F_{n+1}^{(1)} & F_{n+2}^{(1)} & F_{n+3}^{(1)} \\
F_{n+2}^{(1)} & F_{n+3}^{(1)} & F_{n+4}^{(1)}
\end{pmatrix}
=(-1)^n.
\]
\end{lemma}
\proof
By using relation 
\begin{eqnarray}
F_n^{(1)} =F_{n+2} -1
\end{eqnarray}
(that immediately follows from the elementary Fibonacci identity $\sum_{k=0}^{n}F_k=F_{n+2}-1$, $n \ge0$)
and Lemma \ref{L:1} we have
\[
\det
		\begin{pmatrix}
F_{n }^{(1)} & F_{n+1}^{(1)} & F_{n+2}^{(1)} \\
F_{n+1}^{(1)} & F_{n+2}^{(1)} & F_{n+3}^{(1)} \\
F_{n+2}^{(1)} & F_{n+3}^{(1)} & F_{n+4}^{(1)}
\end{pmatrix}
=
\det
		\begin{pmatrix}
F_{n+2}-1 & F_{n+3}-1 & F_{n+4}-1 \\
F_{n+3}-1 & F_{n+4}-1 & F_{n+5}-1 \\
F_{n+4}-1 & F_{n+5}-1 & F_{n+6}-1
\end{pmatrix}
=
(-1)^n\qed
\]
%%%%%%%%%%%%%%%%%%%%%%%%%%%%%%%%%%%%%%

\begin{theorem} \label{CassiniDeterminantGeneralized}
For the sequence $(F_k^{(r)})_{k \ge 0}$, $r \in \mathbb{N}$ and $n \in \mathbb{Z}$ a determinant of a matrix $A_{r,n}$ takes values $\pm 1$,
\begin{eqnarray}
\det (A_{r,n})=(-1)^{n+ \big \lfloor \frac{r+3}{2} \big \rfloor}.
\end{eqnarray}
\end{theorem}

\begin{proof}
Using elementary transformations on matrix and Lemma \ref{Qdeterminant} we get
\begin{eqnarray*}
%&&\det(A_{r,0})=\\
%&=&
\det(A_{r,0})&=&
\det
 \begin{pmatrix}
  F_{0}^{(r)} & F_{1}^{(r)}  &  \cdots & F_{r-2}^{(r)} & F_{r-1}^{(r)} \\
  F_{1}^{(r)} & F_{2}^{(r)} &  \cdots& F_{r-1}^{(r)}  & F_{r}^{(r)}  \\
 \vdots    & \vdots  &     & \vdots  & \vdots  \\
  F_{r-2}^{(r)} & F_{r-1}^{(r)} &    \cdots & F_{2r-4}^{(r)} &    F_{2r-3}^{(r)}\\
   F_{r-1}^{(r)} & F_{r}^{(r)} &    \cdots & F_{2r-3}^{(r)} &    F_{2r-2}^{(r)}
 \end{pmatrix}\\
&=& -
\det
 \begin{pmatrix}
  0 & F_{0}^{(r)}  &  \cdots & F_{r-3}^{(r)} & F_{r-2}^{(r)} \\
  F_{0}^{(r)} & F_{1}^{(r)} &  \cdots & F_{r-2}^{(r)}  & F_{r-1}^{(r)}  \\
 \vdots    & \vdots  &     & \vdots  & \vdots \\
  F_{r-3}^{(r)} & F_{r-2}^{(r)} &    \cdots & F_{2r-5}^{(r)}  &   F_{2r-4}^{(r)}\\
   F_{r-2}^{(r)} & F_{r-1}^{(r)} &    \cdots & F_{2r-4}^{(r)}  &   F_{2r-3}^{(r)}
 \end{pmatrix}\\
&=&
(-1)^{r}
\det
 \begin{pmatrix}
  0 & 0  &  \cdots & 0 & 1 \\
  0 & 0 &  \cdots & 1 & F_{2}^{(r)}  \\
 \vdots    & \vdots  &   \iddots  & \vdots  & \vdots \\
   0 & 1 &    \cdots &   F_{r-2}^{(r)} & F_{r-1}^{(r)}\\
   1 & F_{2}^{(r)} &    \cdots &   F_{r-1}^{(r)} & F_{r}^{(r)}
 \end{pmatrix}\\
&=&
(-1)^{r} (-1)^{ \lfloor \frac{r+2}{2} \rfloor }
\det
 \begin{pmatrix}
   1 & F_{2}^{(r)} &    \cdots &   F_{r-1}^{(r)} & F_{r}^{(r)}\\
   0 & 1 &    \cdots &   F_{r-2}^{(r)} & F_{r-1}^{(r)}\\
 \vdots    & \vdots  &   \ddots  & \vdots  & \vdots \\
  0 & 0 &  \cdots & 1 & F_{2}^{(r)}  \\
  0 & 0  &  \cdots & 0 & 1 \\
 \end{pmatrix}\\
&=&  (-1)^{ \lfloor \frac{r+3}{2} \rfloor }.
\end{eqnarray*}
According to Theorem \ref{thmQmatrix} we obtain
\begin{eqnarray*}
\det (A_{r,n} ) &=& \det(Q_{r+2})^n \det(A_{r,0}) = (-1)^n \det(A_{r,0}) \\
&=& (-1)^n (-1)^{ \lfloor \frac{r+3}{2} \rfloor } = (-1)^{n+ \lfloor \frac{r+3}{2} \rfloor } ,
\end{eqnarray*}
which completes the statement of the theorem.

\end{proof}

%%%%%%%%%%%%%%%%%%%%%%%%%%%%%%%%%%%%%%
\noindent 
Let $M=M^{(m,n,r)}$ be a matrix with $M_{i,j}=F_{n+i+j-2}^{(r)}$, $1\leq i,j\leq m$.
Theorem \ref{CassiniDeterminantGeneralized} can be restated as
\[
\det\left(M^{(n,r,r+2)}\right)=(-1)^{n+ \big \lfloor \frac{r+3}{2} \big \rfloor}.
\]
At the end let us show that for $m>r+2$ the following equlity holds:
\begin{eqnarray} \label{eqZeroDet}
\det\left(M^{(m,n,r)}\right)=0.
\end{eqnarray}

The proof of (\ref{eqZeroDet}) consists of performing elementary transformations on $M^{(m,n,r)}$ leading to a matrix having one column consisting of zeroes.

\noindent
Take a look at the $i-th$ row of $M^{(m,n,r)}$:
\[
[F_{n+i-1}^{(r)}\ F_{n+i}^{(r)}\ F_{n+i+1}^{(r)}\ \cdots\ F_{n+i+j-2}^{(r)}\ \cdots\ F_{n+i+m-3}^{(r)}\ F_{n+i+m-2}^{(r)}].
\]
Using (\ref{obviousRec}) and subtracting $j-th$ element from $(j+1)-st$ for $j=m-1,m-2,\ldots,2,1$ 
(thus simulating subtracting a column $j$ from column $j+1$ in a matrix $M^{(m,n,r)}$)
we get
\[
[F_{n+i-1}^{(r)}\ F_{n+i}^{(r-1)}\ F_{n+i+1}^{(r-1)}\ \cdots\ F_{n+i+j-2}^{(r-1)}\ \cdots\ F_{n+i+m-3}^{(r-1)}\ F_{n+i+m-2}^{(r-1)}].
\]
We can repeat the process for $j=m-1,m-2,\ldots, 3,2$ and get
\[
[F_{n+i-1}^{(r)}\ F_{n+i}^{(r-1)}\ F_{n+i+1}^{(r-2)}\ \cdots\ F_{n+i+j-2}^{(r-2)}\ \cdots\ F_{n+i+m-3}^{(r-2)}\ F_{n+i+m-2}^{(r-2)}].
\]
After repeating the process $r-th$ time (for $j=m-1,m-2,\ldots, r$), we get
\[
[F_{n+i-1}^{(r)}\ F_{n+i}^{(r-1)}\ F_{n+i+1}^{(r-2)}\ \cdots\ F_{n+i+r-2}^{(1)}\ F_{n+i+r-1}^{}\ \cdots\ F_{n+i+m-2}^{}].
\]
Since $m>r+2$ we have $n+i+r-1\leq n+i+m-4$ so the above row contains 
\[
[\cdots F_{n+i+r-1}^{}\ F_{n+i+r}^{}\ F_{n+i+r+1}^{} \cdots]
\]
at positions $r-1$, $r$ and $r+1$. Subtracting first two elements from the third, we get
\[
[\cdots F_{n+i+r-1}^{}\ F_{n+i+r}^{}\ 0 \cdots].
\]
That way we arrive at a matrix with column consisiting of zeroes whose determinant is therefore zero.

%%%%%%%%%%%%%%%%%%%%%%%%%%%%%%%%%%%%%


\begin{thebibliography}{6}

\bibitem{CaZh} N.N. Cao, F. Z. Zhao, Some Properties of Hyperfibonacci and Hyperlucas Numbers, J. Integer Seq. 13, article 10.8.8 1-11 (2010) 

\bibitem{DiMe} A. Dill, I. Mez\H o, A symmetric algorithm for hyperharmonic and Fibonacci numbers, Appl. Math. Comput. 206, 942-951 (2008) 

\bibitem{GKP} R. L. Graham, D. E. Knuth, O. Patashnik, Concrete Mathematics, Addison-Wesley (1994)

\bibitem{KrOl} C. Krattenthaler, A. M. Oller-Marc\'en, A Determinant of Generalized Fibonacci Numbers, J. Combin. Number Theory 5(2), article 2 1--7 (2003)

\bibitem{LeSh} G.Y. Lee, S.G. Lee, H.G. Shin, On the K-Generalized Fibonacci Matrix $Q_k^*$, Linear Algebra Appl. 251, 73-88 (1997)

\bibitem{Mile} E. Miles, Generalized Fibonacci numbers and related matrices, Amer. Math. Monthly 67, 745-752 (1960) 

%\bibitem{Stak} A. P. Stakhov, Fibonacci matrices, a generalization of the Cassini formula and a new coding theory, Chaos Solitons Fractals 30(1), 55-66 (2006) 

\bibitem{Voll} N. G. Voll, The Cassini Identity and Its Relatives, Fibonacci Quart. 48(3), 197-201 (2010)

\bibitem{Vajda} S. Vajda, Fibonacci and Lucas numbers, and the golden section, John Wiley  Sons, New York (1989)

\bibitem{WeZe} M. Werman, D. Zeilberger, A Bijective proof of Cassini's Fibonacci identity, Discrete Math. 58, 109 (1986)


\end{thebibliography}
\end{document}